\documentclass[12pt]{amsart}
\usepackage{amssymb, amsfonts, amsthm} 
\usepackage{graphicx}

\newtheorem{thm}{Theorem}[section]

\newcommand{\bthm}{\begin{thm}}
\newcommand{\ethm}{\end{thm}}

\newtheorem{theorem}[equation]{Theorem}
\newtheorem{lemma}[equation]{Lemma}

\numberwithin{equation}{section}

\theoremstyle{definition}
\newtheorem{definition}[equation]{Definition}
\newtheorem{remark}[equation]{Remark}
\newtheorem{example}[equation]{Example}

\newtheorem{mysubsection}[equation]{\bf }

\newcommand{\be}{\begin{equation}}
\newcommand{\ee}{\end{equation}}

\DeclareMathOperator{\diam}{diam}

\newcommand{\ds}{\displaystyle}

\newcommand{\comment}[1]{}

\newcounter{minutes}
\divide\time by 60
\newcounter{hours}
\multiply\time by 60 \addtocounter{minutes}{-\time}

\begin{document}

\title[Freely quasiconformal maps in Banach spaces]
{Quasiconformal maps with bilipschitz or identity boundary values in Banach spaces}
\author[]{Y. Li}
\address{Yaxiang. Li, Department of Mathematics,
Hunan Normal University, Changsha,  Hunan 410081, People's Republic
of China} \email{yaxiangli@163.com}
\author[]{M. Vuorinen}
\address{Matti. Vuorinen, Department of Mathematics and Statistics, University of Turku,
FIN-20014 Turku, Finland}
\email{vuorinen@utu.fi}

\author{X. Wang
${}^{~\mathbf{*}}$}
\address{Xiantao. Wang, Department of Mathematics,
Hunan Normal University, Changsha,  Hunan 410081, People's Republic
of China} \email{xtwang@hunnu.edu.cn}


\date{}
\keywords{Uniform domain, FQC map, quasisymmetric, bilipschitz boundary values, H\"older condition.\\
${}^{\mathbf{*}}$ Corresponding author\\
The research was supported by the Academy of Finland, Project 2600066611, and by NSF of
China (No. 11071063) and Hunan Provincial Innovation Foundation For
Postgraduate.}



\def\thefootnote{}
\footnotetext{ \texttt{\tiny File:~\jobname .tex,
          printed: \number\year-\number\month-\number\day,
          \thehours.\ifnum\theminutes<10{0}\fi\theminutes }
} \makeatletter\def\thefootnote{\@arabic\c@footnote}\makeatother

\begin{abstract}
Suppose that $E$ and $E'$ denote real Banach spaces with dimension
at least $2$ and that $D\varsubsetneq E$ and $D'\varsubsetneq E'$ are uniform domains with homogeneously dense boundaries. We consider the class of all $\varphi$-FQC (freely $\varphi$-quasiconformal) maps of $D$ onto $D'$ with bilipschitz boundary values. We show that the maps of this class are $\eta$-quasisymmetric.   As an application, we show that if $D$ is bounded, then maps of this class satisfy a two sided H\"older condition. Moreover, replacing the class $\varphi$-FQC by the smaller class of $M$-QH maps, we show that $M$-QH maps with bilipschitz boundary values are bilipschitz.
Finally, we show that if $f$ is a $\varphi$-FQC map which maps $D$ onto itself with identity boundary values, then there is a constant $C\,,$ depending only on the function $\varphi\,,$ such that for all $x\in D$, the quasihyperbolic distance satisfies $k_D(x,f(x))\leq C$.
\end{abstract}

\maketitle\pagestyle{myheadings} \markboth{}{Quasiconformal maps with bilipschitz or identity boundary values in Banach spaces}


\section{Introduction and main results}
\label{intro}

Many results of classical function theory have their counterparts in the context of quasiconformal maps in the Euclidean $n$-dimensional space $\mathbb{R}^n$.  J. V\"ais\"al\"a \cite{Vai6-0, Vai6, Vai5}
has developed a theory of quasiconformality in the Banach space case
which
differs from the finite dimensional theory in many respects because tools such as conformal invariants and measures of sets are no longer available. These classical tools are replaced by
fundamental objects from metric space geometry such as curves, their lengths, and approximately length minimizing curves.
 V\"ais\"al\"a used these notions in the setup of several
 metric space structures on the same underlying Banach space and
developed effective methods based on these basic notions. In addition to the norm metric
 he considered two hyperbolic type metric structures,
 the quasihyperbolic metric and the distance ratio metric. The quasihyperbolic metric $k_D$ of a domain $D$ has a key role as quasiconformality is defined in terms of it in the Banach space case. Only recently some basic properties of quasihyperbolic
metric have been studied: the convexity of quasihyperbolic balls was
studied by R. Kl\'en \cite{k,k2}, A. Rasila and J. Talponen  \cite{rt,krt}, V\"ais\"al\"a \cite{Vai6'}. Rasila and Talponen
also proved the smoothness of quasihyperbolic geodesics in \cite{rt2}
applying now stochastic methods.

Given domains $D, D'$ in Banach spaces $E$ and $E'$, respectively, our basic problem is to study the class of homeomorphisms $f\in QC^L_{\varphi}(D,D')$, where
\begin{equation}\label{intro-eq-1}
QC^L_{\varphi}(D,D') = \{f: \overline{D}\to \overline{D}'\; { \rm homeo}\;\Big| f|_{D} \;{\rm is}\; {\rm a}\;\varphi{\rm -FQC}\;{\rm map}\; {\rm and}\; f|_{\partial D}\; {\rm is}\; L{\rm -bilipschitz}\}\, .
\end{equation}
For the definition of $\varphi$-FQC and $L$-bilipschitz maps see Section \ref{sec-2}. The class $QC^L_{\varphi}(D,D')$ is very wide  and many particular cases of interest are obtained by choosing $D, D', \varphi, L$ in a suitable way as we will see below.

Our first result deals with the case when both $D$ and $D'$ are uniform domains. 
In this case we prove that the class \eqref{intro-eq-1} consists of quasisymmetric maps. More precisely, we prove the following theorem.

\begin{theorem}\label{thm1.2}
Let $D\subsetneq E$, $D'\subsetneq E'$ be $c$-uniform domains. If $f\in QC^L_{\varphi}(D,D')$,
then $f$ is $\eta$-QS in $\overline{D}$ with $\eta$ depending on $c$, $L$ and $\varphi$ only.
\end{theorem}

Applying this result to the case of a bounded domain $D$ we obtain the second result. Recall that in the case of ${\mathbb R}^n$ results of this type have been proved by R. N\"akki and B. Palka \cite{np}. For the
definitions, see Section \ref{sec-2}.

\begin{theorem}\label{thm1.3}
Let $D\subsetneq E$, $D'\subsetneq E'$ be $c$-uniform domains. If $f\in QC^L_{\varphi}(D,D')$ and $D$ is bounded, then  for all $x,y\in D$, $$\frac{|x-y|^{1/{\alpha}}}{C}\leq|f(x)-f(y)|\leq C|x-y|^{\alpha},$$ where $C\geq 1$ and $\alpha\in(0,1)$ depend on $c$, $L$, $\varphi$ and $\diam(D)$.
\end{theorem}

Our third result concerns the case when both $D$ and $D'$ are uniform domains and $\varphi(t)=Mt$ for some fixed $M \ge 1\,.$ We also require a density condition of the boundary of a domain. This
$(r_1,r_2)$-HD condition will be defined in Section 2.

\begin{theorem}\label{thm1.4}
Let $D\subsetneq E$, $D'\subsetneq E'$ be $c$-uniform domains and the boundary of $D$ be $(r_1,r_2)$-HD. If $f\in QC^L_{\varphi}(D,D')$ with $\varphi(t)=Mt$, then $f$ is $M'$-bilipschitz in $\overline{D}$, where $M'$ depends only on $c$, $r_1$, $r_2$, $L$ and $M$.
\end{theorem}

Our fourth result deals with the case when $D=D'$, $L=1$ and, moreover, the boundary mapping
 $f|_{\partial D}:\partial D\to\partial D$ is the identity. This problem has been studied very recently in  \cite{MV,M2,VZ}.
Originally, the problem was motivated by Teichm\"uller's
work on plane quasiconformal maps \cite{K,T} and then extended to
the higher dimensional case by several authors:  \cite{AV},
 \cite{M2},  \cite{MV,VZ}. Our result is as follows.

\begin{theorem}\label{thm1.1}
Let $D\subsetneq E$ be a $c$-uniform domain with $(r_1,r_2)$-HD  boundary.
If $f$ is a $\varphi$-FQC map which maps $D$ onto itself with identity boundary values, then for all $x\in D$, $$k_D(x, f(x))\leq C,$$ where $C$ is a constant depending on $r_1$, $r_2$, $c$ and $\varphi$ only.

\end{theorem}

For the case $n=2$, when $D$ is the unit disk, the sharp bound is
due to Teichm\"uller \cite{K, T}. For the case of unit ball in $\mathbb{R}^n, n\ge 2,$ nearly sharp results appear in \cite{MV,VZ}. In both
of these cases one uses the hyperbolic metric in place of the
quasihyperbolic metric.

We do not know whether there are sharp results for the Banach spaces, too. For instance, it is an open problem whether Theorem \ref{thm1.1} could be refined for the case $D=\mathbb{B}$, the unit ball, to the effect that $C\rightarrow 0$ when $\varphi$ approaches
the identity map.

The organization of this paper is as follows. In Section
\ref{sec-4}, we will prove Theorems \ref{thm1.2}, \ref{thm1.3},
\ref{thm1.4} and \ref{thm1.1}. In Section \ref{sec-2}, some
preliminaries are stated.


\section{Preliminaries}\label{sec-2}

\mysubsection{\bf Notation.} We adopt mostly the standard notation and terminology from  V\"ais\"al\"a \cite{Vai6-0,Vai5}. We always use $E$ and $E'$ to denote real
Banach spaces with dimension at least $2$. The norm of a vector $z$
in $E$ is written as $|z|$, and for every pair of points $z_1$,
$z_2$ in $E$, the distance between them is denoted by $|z_1-z_2|$,
the closed line segment with endpoints $z_1$ and $z_2$ by $[z_1,
z_2]$. Moreover, we use $\mathbb{B}(x, r)$ to denote the ball with center $x\in E$ and radius $r$ $(> 0)$,
 and its boundary and closure are denoted by
$\mathbb{S}(x,\; r)$ and $\overline{\mathbb{B}}(x,\; r)$,
respectively. In particular, we use $\mathbb{B}$ to denote the unit
ball $\mathbb{B}(0,\; 1)$. The one-point extension of $E$ is the Hausdorff space $\dot{E}=E\cup\{\infty\}$, where the neighborhoods of
$\infty$ are the complements of closed bounded sets of $E$. The boundary $\partial A$ and the closure $\overline{A}$ of a set $A\subset E$ are taken in $\dot{E}$.

\mysubsection{\bf Quasihyperbolic distance and uniform domains.} The {\it quasihyperbolic length}
of a rectifiable arc or a path
$\alpha$ in the norm metric in $D$ is the number (cf.
\cite{Geo,GP,Vai6-0}):

$$\ell_k(\alpha)=\int_{\alpha}\frac{|dz|}{d_{D}(z)},
$$ where $d_D(z)$ denotes the
distance from $z$ to the boundary $\partial D$ of $D$.

For each pair of points $z_1$, $z_2$ in $D$, the {\it quasihyperbolic distance}
$k_D(z_1,z_2)$ between $z_1$ and $z_2$ is defined in the usual way:
$$k_D(z_1,z_2)=\inf\ell_k(\alpha),
$$
where the infimum is taken over all rectifiable arcs $\alpha$
joining $z_1$ to $z_2$ in $D$.

For each pair of points $z_1$, $z_2$ in $D$, the {\it distance ratio
metric} $j_D(z_1,z_2)$ between $z_1$ and $z_2$ is defined by
$$j_D(z_1,z_2)=\log\Big(1+\frac{|z_1-z_2|}{\min\{d_D(z_1),d_D(z_2)\}}\Big).$$

 For all $z_1$, $z_2$ in $D$, we have
(cf. \cite{Vai6-0})

\begin{equation}\label{eq(0000)} k_{D}(z_1, z_2)\geq
\inf\left\{\log\Big(1+\frac{\ell(\alpha)}{\min\{d_{D}(z_1), d_{D}(z_2)\}}\Big)\right\}\geq j_D(z_1, z_2)\end{equation}
$$ \geq
\Big|\log \frac{d_{D}(z_2)}{d_{D}(z_1)}\Big|,$$
where the infimum is taken over all rectifiable curves $\alpha$ in $D$ connecting $z_1$ and $z_2$. Moreover, if $|z_1-z_2|\le d_D(z_1)$, we have
\cite{Vai6-0, Mvo1}
\begin{equation} \label{vu1}
k_D(z_1,z_2)\le \log\Big( 1+ \frac{
|z_1-z_2|}{d_D(z_1)-|z_1-z_2|}\Big).
\end{equation}

Gehring and Palka \cite{GP} introduced the quasihyperbolic metric of
a domain in $\mathbb{R}^n$ and it has been recently used by many authors
 in the study of quasiconformal mappings and related questions \cite{HIMPS, krt, rt} etc.

\begin{definition}   A domain $D$ in $E$ is called $c$-{\it
uniform} in the norm metric provided there exists a constant $c$
with the property that each pair of points $z_{1},z_{2}$ in $D$ can
be joined by a rectifiable arc $\alpha$ in $ D$ satisfying (see \cite{Martio-80, Vai, Vai4})

 \begin{enumerate}
\item\label{wx-4} $\ds\min_{j=1,2}\ell (\alpha [z_j, z])\leq c\,d_{D}(z)
$ for all $z\in \alpha$, and

\item\label{wx-5} $\ell(\alpha)\leq c\,|z_{1}-z_{2}|$,
\end{enumerate}

\noindent where $\ell(\alpha)$ denotes the length of $\alpha$ and
$\alpha[z_{j},z]$ the part of $\alpha$ between $z_{j}$ and $z$. Moreover, $\alpha$ is said to be a {\it uniform arc}.
\end{definition}

In \cite{Vai6},  V\"ais\"al\"a characterized uniform domains as follows.

\medskip
\begin{lemma}\label{pre-lem-1}\;\; {\rm (\cite[Theorem 6.16]{Vai6})}
For a domain $D$, the following are quantitatively equivalent: \begin{enumerate}

\item $D$ is a $c$-uniform domain;
\item $k_D(z_1,z_2)\leq c'\; j_D(z_1,z_2)$
for all $z_1,z_2\in D$;
\item $k_D(z_1,z_2)\leq c'_1\; j_D(z_1,z_2) +d$
for all $z_1,z_2\in D$.\end{enumerate}

\end{lemma}

\medskip

 In the case of domains in $ {\mathbb R}^n \,,$ the equivalence
  of items (1) and (3) in Theorem D is due to Gehring and Osgood \cite{Geo} and the
  equivalence of items (2) and (3) due to Vuorinen \cite{Mvo1}. Many of the
  basic properties of this metric may be found in \cite{Geo, krt, rt, Vai6-0, Vai6}.

In \cite{Vai5}, V\"ais\"al\"a proved the following examples for some special uniform domain.

\begin{example}\;\; {\rm (\cite[Examples 10.4]{Vai5})}(1) Each ball $B\subset E$ is $2$-uniform;

(2) Every bounded convex domain $G\subset E$ is uniform;

(3) Half space $H\subset E$ is $c$-uniform for all $c>2$.

\end{example}


\medskip
\mysubsection{\bf Bilipschitz and FQC maps.}

\begin{definition}  Suppose $G\varsubsetneq E\,,$ $G'\varsubsetneq
E'\,,$ and $M \ge 1\,.$ We say that a homeomorphism $f: G\to G'$ is
{\it $M$-bilipschitz} if
 $$|x-y|/M \leq |f(x)-f(y)|\leq M\,|x-y|$$
for all $x$, $y\in G$,  and {\it $M$-QH} if
 $$k_{G}(x,y)/M\leq k_{G'}(f(x),f(y))\leq M\,k_{G}(x,y)$$
for all $x$, $y\in G$.
\end{definition}

Clearly, if $f$ is $M$-bilipschitz or  $M$-QH, then also $f^{-1}$
has the same property.

\begin{definition}   Let $G\not=E$ and $G'\not=E'$ be metric
spaces, and let $\varphi:[0,\infty)\to [0,\infty)$ be a growth
function, that is, a homeomorphism with $\varphi(t)\geq t$. We say
that a homeomorphism $f: G\to G'$ is {\it $\varphi$-semisolid} if
$$ k_{G'}(f(x),f(y))\leq \varphi(k_{G}(x,y))$$
for all $x$, $y\in G$, and {\it $\varphi$-solid} if both $f$ and $f^{-1}$ satisfy this condition.

We say that $f$ is {\it fully $\varphi$-semisolid}
(resp. {\it fully $\varphi$-solid}) if $f$ is
$\varphi$-semisolid (resp. $\varphi$-solid) on every  subdomain of $G$. In particular,
when $G=E$, the corresponding subdomains are taken to be proper ones. Fully $\varphi$-solid maps are also called {\it freely
$\varphi$-quasiconformal maps}, or briefly {\it $\varphi$-FQC maps}.
\end{definition}

Clearly, if $f$ is freely
$\varphi$-quasiconformal, then so is $f^{-1}\,.$

If $E=\mathbb{R}^n=E'$, then $f$ is $FQC$ if and only if $f$ is
quasiconformal (cf. \cite{Vai6-0}). See \cite{Vai1, Mvo1} for definitions and
properties of $K$-quasiconformal maps, or briefly $K$-QC maps.

\mysubsection{\bf  Quasisymmetric and quasim\"obius maps.}  Let $X$ be a metric space and $\dot{X}=X\cup \{\infty\}$. By a
triple in $X$ we mean an ordered sequence $T=(x,a,b)$ of three
distinct points in $X$. The ratio of $T$ is the number
$$\rho(T)=\frac{|a-x|}{|b-x|}.$$ If $f: X\to Y$ is  an injective
map, the image of a triple  $T=(x,a,b)$  is the triple
$fT=(fx,fa,fb)$.


\begin{definition}
Let $X$ and $Y$ be two metric spaces, and let
$\eta: [0, \infty)\to [0, \infty)$ be a homeomorphism.
An embedding $f: X\to Y$ is said to be {\it $\eta$-quasisymmetric}, or briefly $\eta$-$QS$, if $\rho(f(T))\leq \eta(\rho(T))$  for each triple $T$ in $X$.
\end{definition}

It is known that an embedding $f: X\to Y$ is $\eta$-$QS$  if
and only if $\rho(T)\leq t$ implies that $\rho(f(T))\leq \eta(t)$
for each triple $T$ in $X$ and $t\geq 0$ (cf. \cite{TV}).

A quadruple in $X$ is an ordered sequence $Q=(a,b,c,d)$ of four
distinct points in $X$. The cross ratio of $Q$ is defined to be the
number
$$\tau(Q)=|a,b,c,d|=\frac{|a-b|}{|a-c|}\cdot\frac{|c-d|}{|b-d|}.$$
Observe that the definition is extended in the well known manner
to the case where one of the points is $\infty$. For example,
$$|a,b,c,\infty|= \frac{|a-b|}{|a-c|}.$$
If $X_0 \subset \dot{X}$ and if $f: X_0\to \dot{Y}$
is an injective map, the image of a quadruple $Q$ in $X_0$ is the
quadruple $fQ=(fa,fb,fc,fd)$.


\begin{definition}   Let $X$ and $Y$ be two metric spaces and let
$\eta: [0, \infty)\to [0, \infty)$ be a homeomorphism.  An embedding $f: X\to Y$ is said to be {\it
$\eta$-quasim\"obius} (cf. \cite{Vai2}), or briefly $\eta$-$QM$, if the inequality $\tau(f(Q))\leq \eta(\tau(Q))$ holds for each
quadruple $Q$ in $X$.
\end{definition}

Observe that if $\infty\in X$ and if $f:X\to Y$ is $\eta$-quasim\"obius with $f(\infty)=\infty$, then $f$ is $\eta$-quasisymmetric (see \cite[6.18]{Vai5}). Conversely, the following result holds.

\medskip

\begin{lemma}\label{pre-lem-2}\;\;$($\cite[Theorem 3.12]{Vai2}$)$
Suppose that $X$ and $Y$ are bounded spaces, that $\lambda>0$, that $z_1,z_2,z_3\in X$, and that $f:X \to Y$ is $\theta$-quasim\"obius such that
$$|z_i-z_j|\geq \diam(X)/{\lambda}\, \mbox{,}\, |f(z_i)-f(z_j)|\geq\diam(Y)/{\lambda}$$ for $i\neq j$. Then $f$ is $\eta$-quasisymmetric with $\eta=\eta_{\theta,\lambda}$.
\end{lemma}

\medskip
Concerning the relation between the class of uniform domains and
quasim\"obius maps, V\"ais\"al\"a proved the following result.

\medskip

\begin{lemma}\label{pre-lem-3}\;\;{\rm (\cite[Theorems 11.8 and 11.15]{Vai5})}
Suppose that $D\varsubsetneq E$ and $D'\varsubsetneq E'$, that $D$ and $D'$
are $c$-uniform domain, and that $f:D\to D'$ is a $\varphi$-FQC map. Then
 $f$ extends to a homeomorphism $\overline{f}: \overline{D}\to
\overline{D}'$ and $\overline{f}$ is $\theta_1$-QM in $\overline{ D}$.
\end{lemma}

\medskip

Finally we introduce the concept of homogeneous density from \cite{TV}.

\begin{definition}{\rm (\cite[Definition 3.8]{TV})}
A space $X$ is said to be {\it homogeneously dense}, abbreviated HD, if there are numbers $r_1$, $r_2$ such that $0<r_1\leq r_2<1$ and such that for each pair of points $a,b\in X$ there is $x\in X$ satisfying the condition $$r_1|b-a|\leq|x-a|\leq r_2|b-a|.$$ We also say that $X$ is $(r_1,r_2)$-HD or simply $r$-HD, where $r=(r_1,r_2)$.
\end{definition}

By the definition, obviously, a HD space has no isolated point. And for all $0<r_1\leq r_2<1$, every connected domain is $(r_1,r_2)$-HD, $[0,1]\cup [2,3]$ is $(\frac{1}{6},\frac{1}{4})$-HD (see \cite{TV}). Particularly, a finite union of connected nondegenerate sets (i.e. the set is not a point) is $(r_1,r_2)$-HD with some constants $0<r_1\leq r_2<1$.


For a HD space, Tukia and V\"ais\"al\"a  proved the following properties in \cite{TV}.

\medskip
\begin{lemma}\label{pre-lem-4}{\rm (\cite[Lemma 3.9]{TV})}
$(\textit{1})$ Let $X$ be $(r_1,r_2)$-HD and let $m$ be a positive integer. Then $X$ is $(r_1^m,r_2^m)$-HD.

$(\textit{2})$ Let $X$ be $r$-HD and let $f:X\to Y$ be $\eta$-QS. Then $fX$ is $\mu$-HD, where $\mu$ depends only on $\eta$ and $r$.
\end{lemma}

\medskip
Moreover, we prove the following property.

\begin{lemma}\label{lem-2-2}
Let $D\subsetneq E $ be a domain with $(r_1,r_2)$-HD boundary and let $x\in D$. Then for all $x_0\in \partial D$ with
$|x-x_0|\leq 2d_D(x)$ there exists some point
$x_1\in \partial D$ such that
\begin{equation}\label{eq-th-ll}\frac{1}{2}d_D(x)
\leq |x_0-x_1|\leq \big(2+\frac{17}{2r_1}\big) d_D(x).
\end{equation}
\end{lemma}

\begin{proof} By Lemma \ref{pre-lem-4} we may assume that
$0<r_1\leq r_2<\frac{1}{3}$. For example, if
$r_2\geq \frac{1}{3}$, then there exists a positive integer $m$ such that $r_2^m<\frac{1}{3}$. In fact we can choose $m-1$ to be the integer part of $\log_{r_2}\frac{1}{3}$, and by Lemma \ref{pre-lem-4} the  $(r_1,r_2)$-HD property of $\partial D$ implies that $\partial D$ is $(r_1^m,r_2^m)$-HD with $r_2^m<\frac{1}{3}$.

For a given $x\in D$, let $x_0\in \partial D$ be such that $|x-x_0|\leq 2d_D(x)\,.$ We divide the proof into three cases.

{\em Case I}: $\partial D\subset \overline{\mathbb{\mathbb{B}}}\big(x,\frac{5}{2}d_D(x)\big)$.

Obviously, $D$ is bounded. Let $x_1\in \partial D$ be such that $|x_0-x_1|\geq \frac{1}{3}\diam(D)$. Then
$$\frac{2}{3}d_D(x)\leq |x_0-x_1|\leq 5d_D(x),$$
which shows that $x_1$ is the desired point
and satisfies \eqref{eq-th-ll}.

{\em Case II}: $\partial D\cap  \Big(\mathbb{B}\big(x, \frac{1}{r_1}d_D(x)\big)\setminus \overline{\mathbb{B}}\big(x,\frac{5}{2}d_D(x)\big)\Big)\neq \emptyset$.

Let $x_2\in \partial D\cap \mathbb{\mathbb{B}}\big(x, \frac{1}{r_1}d_D(x)\big)\setminus \overline{\mathbb{B}}\big(x,\frac{5}{2}d_D(x)\big)$. Then $$|x_0-x_2|\geq |x_2-x|-|x-x_0|\geq \frac{1}{2}d_D(x)$$ and $$|x_0-x_2|\leq |x_0-x|+|x-x_2|\leq \big(\frac{1}{r_1}+2\big)d_D(x).$$

Obviously, $x_2$ is the needed point.

{\em Case III}: $\partial D\cap \Big( \mathbb{B}\big(x, \frac{1}{r_1}d_D(x)\big)\setminus \overline{\mathbb{B}}\big(x,\frac{5}{2}d_D(x)\big)\Big)= \emptyset$.

Let $\omega=\partial D \cap (E\setminus \mathbb{B}\big(x, \frac{1}{r_1}d_D(x))\big)$ and $d_1$ denote the distance from $\omega$ to $\mathbb{B}\big(x, \frac{1}{r_1}d_D(x)\big)$, i.e., $d_1=d\Big(\omega, \mathbb{B}\big(x, \frac{1}{r_1}d_D(x)\big)\Big)$. If $d_1=0$, let $x_3\in \omega$ be such that $d(x_3, \mathbb{B}\big(x, \frac{1}{r_1}d_D(x))\big)\leq \frac{1}{2}d_D(x)$. Hence $$(\frac{1}{r_1}-2)d_D(x)\leq |x_0-x_3|\leq |x_0-x|+|x-x_3|\leq (\frac{1}{r_1}+\frac{5}{2})d_D(x).$$ So $x_3$ is the desired point.

On the other hand, if $d_1>0$,  
let $x_4\in \omega$ be such that \begin{equation}\label{lem-2-eq1}d(x_4, \mathbb{B}\big(x, \frac{1}{r_1}d_D(x))\big)\leq \frac{3}{2}d_1.\end{equation}  We claim that  the point $x_4$ satisfies \eqref{eq-th-ll}. To see this, we first prove
\begin{equation}\label{eq-th-ss} d_1<\frac{5}{r_1}d_D(x).
\end{equation}
Suppose on the contrary that $d_1\geq \frac{5}{r_1}d_D(x).$ Then by \eqref{lem-2-eq1} there exists some point $u\in \partial D$ such that
\begin{eqnarray*}|u-x_0|&\geq& r_1|x_0-x_4|\geq r_1(|x_4-x|-|x-x_0|)\\&\geq& r_1\big(\frac{6}{r_1}-2\big)d_D(x)= (6-2r_1)d_D(x)\end{eqnarray*}
and
\begin{eqnarray*}|u-x_0|&\leq& r_2|x_0-x_4|\leq r_2(|x_0-x|+|x-x_4|)\\&\leq& r_2\big(2+\frac{1}{r_1}\big)d_D(x)+\frac{3r_2}{2}d_1\leq d_1,
\end{eqnarray*}
which shows that $u\in \partial D\cap \Big(\mathbb{B}\big(x, \frac{1}{r_1}d_D(x)\big)\setminus \overline{\mathbb{B}}\big(x,\frac{5}{2}d_D(x)\big)\Big)$. This is a contradiction. Hence \eqref{eq-th-ss} holds.

By \eqref{eq-th-ss}, we have $$\big(\frac{1}{r_1}-2\big)d_D(x)\leq|x_1-x_0|\leq \big(2+\frac{1}{r_1}\big)d_D(x)+\frac{3}{2}d_1\leq \big(2+\frac{17}{2r_1}\big)d_D(x).$$
Hence the point $x_4$ has the required properties, and so the proof of the lemma is complete.

\end{proof}

\medskip

The discussions in the case III also follows from \cite[Lemma 11.7]{H}.


\medskip

\section{Proofs of Theorems \ref{thm1.2}, \ref{thm1.3}, \ref{thm1.4} and \ref{thm1.1}} \label{sec-4}

For convenience, in the following, we always assume that $x$, $y$, $z$, $\ldots$
denote points in $D$ and $x'$, $y'$, $z'$, $\ldots$
the images in $D'$ of $x$, $y$, $z$, $\ldots$
under $f$, respectively.
\medskip
We start with some known results that are necessary for the following proofs.

\begin{lemma}\label{proof-lem-1}\;\;{\rm (\cite[Lemma 2.5]{Vai6-0})}
Suppose that $x,y\in D\neq E$ and that either $|x-y|\leq \frac{1}{2}d_D(x)$ or $k_D(x,y)\leq 1$. Then $$\frac{1}{2}\frac{|x-y|}{d_D(x)}\leq k_D(x,y)\leq 2\frac{|x-y|}{d_D(x)}.$$
\end{lemma}

\medskip

\begin{lemma}\label{proof-lem-2}\;\;{\rm (\cite[Lemma 2.6]{Vai6-0})}
Suppose that $X$ is connected, that $f:X\to Y$ is $\eta$-quasisymmetric, and that $A\subset X$ is bounded. Then $f|_A$ satisfies a two-sided H\"older condition $$|x-y|^{{1}/{\alpha}}/M\leq |fx-fy|\leq M|x-y|^{\alpha}\;\;\;\;\; for \; x,y \in A,$$ where $\alpha=\alpha(\eta)\leq1$ and $M=M(\eta, d(A),d(fA))\geq 1.$
\end{lemma}

\medskip

\mysubsection{\bf The proof of Theorem \ref{thm1.2}}


Since $f:\partial D\to \partial D'$ is $L$-bilipschitz, we know that the boundedness of $D$ (resp. $D'$) implies the boundedness of $D'$ (resp. $D$). In fact, suppose on the contrary that $D$ is bounded and $D'$ is unbounded. Then let $w_1', w_2'\in \partial D'$ such that $|w_1'-w_2'|\geq 4L \diam(D)$. Then we have
$$\diam(D)\geq \frac{1}{2}|w_1-w_2|\geq \frac{1}{2L}|w_1'-w_2'|\geq 2\diam(D),$$
which is a contradiction.

If $D$ is unbounded, then $\infty\in \partial D$, by auxiliary inversions we normalize the situation such that $f(\infty)=\infty.$ Hence by Lemma \ref{pre-lem-3}, $f$ is $\eta$-QS in $\overline{D}$ with $\eta$ depending on $c$, $L$ and $\varphi$.

In the following, we assume that $D$ is bounded. Then \begin{equation}\label{thm-1-4}\frac{1}{4L}\diam(D)\leq\diam(D')\leq 4L\diam(D).\end{equation}
Let $z_1, z_2\in \partial D$ be such that $|z_1-z_2|\geq \frac{1}{2}\diam(D)$ and let $z_3'\in \partial D' $ be such that $$\min\{|z_1'-z_3'|,|z_2'-z_3'|\}\geq \frac{1}{6}\diam(D').$$ Then by \eqref{thm-1-4}, we have $$|z_1'-z_2'|\geq \frac{1}{L}|z_1-z_2|\geq \frac{1}{2}\diam(D)\geq \frac{1}{8L^2}\diam(D')$$ and
$$\min\{|z_1-z_3|,|z_2-z_3|\}\geq\frac{1}{L}\min\{|z_1'-z_3'|,|z_2'-z_3'|\}\geq\frac{1}{24L^2}\diam(D),$$
which, in combination with Lemma \ref{pre-lem-2} and Lemma \ref{pre-lem-3}, shows that $f$ is $\eta$-QS in $\overline{D}$ with $\eta$ depending on $c$, $L$ and $\varphi$.\qed

\mysubsection{\bf The proof of Theorem \ref{thm1.3}}

The proof of Theorem \ref{thm1.3} easily follows from Theorem \ref{thm1.2} and Lemma \ref{proof-lem-2}.\qed
\medskip

\medskip


%


\medskip

In the remaining part of this paper, we always
assume that $D$ and $D'$ are $c$-uniform subdomains  in $E$ and $E'$, respectively,  that the boundary of $D$ is $(r_1,r_2)$-homogeneously dense, that $f: D\to D'$ is a
$\varphi$-FQC map, and that $f$ extends to a homeomorphism $\overline{f}:
\overline{D}\to \overline{D'}$ such that $\overline{f}:\partial
D\to \partial D'$ is $L$-bilipschitz.\medskip

We first show that 
the following lemma holds.

\begin{lemma}\label{lem-1}
There is a constant $M_1=M_1(c,L,\varphi,r_1,r_2)$ such that for given $x\in D$ the following hold:\\
$(1)$ For $x_0\in \partial D$ with $|x-x_0|\leq 2d_D(x)$, we have $$|x_0'-x'|\leq
M_1d_{D}(x).$$
$(2)$ For all $x_1\in \partial D$, we have \begin{equation}\label{eq-lem-ls}\frac{1}{2(2L+M_1)}|x_1-x|\leq|x'_1-x'|\leq 2(2L+M_1)|x_1-x|.\end{equation}
\end{lemma}

\begin{proof} We first prove $(1)$.

 For a fixed $x\in D$, let $x_0\in \partial D$ be such that $|x-x_0|\leq 2d_D(x)$.
Let $x_2$ be
the intersection point of $\mathbb{S}(x, \frac{1}{2}d_{D}(x))$
with $[x_0, x]$. Then by \eqref{vu1} we have
$$k_{D}(x_2,x)\leq
\log\Big(1+\frac{|x-x_2|}{d_{D}(x)-|x-x_2|}\Big)=\log 2,$$ which implies that

$$\log\frac{|x'_2-x'|}{|x'_2-x'_0|}\leq k_{D'}(x'_2,x')\leq \varphi(k_{D}(x_2,x))=\varphi(\log 2).$$ Hence
\begin{equation}\label{lem-1-0}|x'_2-x'|\leq e^{\varphi(\log 2)} |x'_2-x'_0|,\end{equation} \noindent and so
\begin{equation}\label{lem-1-1}|x'-x'_0|\leq|x'-x'_2|+|x'_2-x'_0|
\leq(e^{\varphi(\log 2)}+1)|x'_2-x'_0|.\end{equation}

Since $\partial D$ is $(r_1,r_2)$-HD,  we see from Lemma \ref{lem-2-2} that there must exist some point $x_3\in  \partial D$ such that \begin{equation}\label{lem-1-16}\frac{1}{2}d_D(x)\leq |x_3-x_0|\leq \big(2+\frac{17}{2r_1}\big)d_D(x).\end{equation} Hence

\begin{equation}\label{lem-1-6'}|x-x_3|\leq |x-x_0|+|x_0-x_3|\leq
\big(4+\frac{17}{2r_1}\big)d_D(x)\end{equation}

\noindent and
\begin{equation}\label{lem-1-6}\frac{1}{2L}d_{D}(x)\leq
\frac{1}{L}|x_3-x_0|\leq |x'_3-x'_0|\leq L|x_3-x_0|\leq L\big(2+\frac{17}{2r_1}\big)d_{D}(x).\end{equation}
By  Lemma \ref{pre-lem-3} we see that $f^{-1}$ is
$\theta$-quasim\"obius in $\overline{D}$, where
$\theta=\theta(c,\varphi)$. It follows from (\ref{lem-1-0}),   (\ref{lem-1-16}), (\ref{lem-1-6'}) and
 (\ref{lem-1-6}) that
\begin{eqnarray*}\frac{1}{6\big(4+\frac{17}{2r_1}\big)}&\leq&\frac{|x_3-x_0|}{|x_2-x_0|}\cdot\frac{|x_2-x|}{|x-x_3|}\leq
\theta
\Big(\frac{|x'_3-x'_0|}{|x'_2-x'_0|}\cdot\frac{|x'_2-x'|}{|x'-x'_3|}\Big)\\&\leq&
\theta\Big(\frac{L\big(2+\frac{17}{2r_1}\big)e^{\varphi(\log 2)}d_D(x)}{|x'-x'_3|}\Big) ,\end{eqnarray*}  which,
together with (\ref{lem-1-1}), shows that
\begin{eqnarray*}|x'-x'_0|&\leq& |x'-x'_3|+|x'_3-x'_0|\leq (\lambda L\big(2+\frac{17}{2r_1}\big)e^{\varphi(\log 2)}+1)d_D(x)\\&\leq & 2\lambda L\big(2+\frac{17}{2r_1}\big)e^{\varphi(\log 2)}d_D(x),\end{eqnarray*} where  $\lambda={1}/{\theta^{-1}(\frac{1}{6\big(4+\frac{17}{2r_1}\big)})}$. 
Thus the proof of $(1)$ is complete by taking $M_1=2\lambda L\big(2+\frac{17}{2r_1}\big)e^{\varphi(\log 2)}$.

Now we are going to prove $(2)$.

We first observe that $f:\partial D\to \partial D'$ is $\eta$-QS with $\eta(t)=L^2t$. Hence Lemma \ref{pre-lem-4} shows that $\partial D'$ is $(\lambda_1,\lambda_2)$-HD with $\lambda_1, \lambda_2$ depending only on $L$, $r_1$ and $r_2.$
Since $f^{-1}$ is also a $\varphi$-FQC map, it is easily seen that  we only need to prove the right hand  side of \eqref{eq-lem-ls}.
For $x\in D$, we let $y_1\in
\partial D$ be such that \begin{equation}\label{lem-2-13}|x-y_1|\leq 2d_D(x).\end{equation}   Then it follows from Lemma \ref{lem-1} $(1)$ that
\begin{equation}\label{lem-2-7}|x'-y'_1|\leq M_1d_D(x)\leq M_1|x-y_1|.\end{equation}

For $x_1\in \partial D$, on one hand, if $|y_1-x_1|\leq 2 |x-y_1|$, then by \eqref{lem-2-13},
\begin{eqnarray*} |x'-x'_1|&\leq&
|x'-y'_1|+|y'_1-x'_1|\leq M_1|x-y_1|+L|y_1-x_1|\\ \nonumber &\leq&(2L+M_1)|x-y_1|\leq 2(2L+M_1)d_D(x)\\
\nonumber &\leq& 2(2L+M_1)|x-x_1|.\end{eqnarray*}

On the other hand, if $|y_1-x_1|>2|x- y_1|$, then we have
$$|x-x_1|>|y_1-x_1|-|x-y_1|>\frac{1}{2}|y_1-x_1|,$$ which, together with
(\ref{lem-2-7}), shows that \begin{eqnarray*} |x'-x'_1|&\leq&
|x'-y'_1|+|y'_1-x'_1|\leq M_1|x-y_1|+L|y_1-x_1|\\ \nonumber &\leq& 2M_1d_D(x)+2L|x-x_1|\leq 2(L+M_1)|x-x_1|.\end{eqnarray*} Hence the proof of \eqref{eq-lem-ls} is complete.\end{proof}

\medskip
\mysubsection{\bf The proof of Theorem \ref{thm1.4}}

Supposing that $f \in QC^L_{\varphi}(D,D')$  is $M$-QH, we show that $f$ is $M'$-bilipschitz from $\overline{D}$ to $\overline{D}'$. Lemma \ref{pre-lem-4} yields that $\partial D'$ is $(\lambda_1,\lambda_2)$-HD with $\lambda_1, \lambda_2$ depending only on $L$, $r_1$ and $r_2$. Then by Lemma \ref{lem-1} and the fact that $``f^{-1}$ is also $M$-QH and a $M$-QH map is a $\varphi$-FQC map with $\varphi(t)=Mt$" we know that it suffices to show that for all $z_1,z_2\in D$, the following holds:
\begin{equation}\label{thm-1-2}|z_1'-z_2'|\leq M'|z_1-z_2|.\end{equation}

Fix $z_1,z_2\in D\,.$ Without loss of generality, we may assume that $$\max\{d_D(z_1),d_D(z_2)\}=d_D(z_1).$$

Consider first the case $|z_1-z_2|\leq \frac{1}{2M}d_D(z_1)\,.$ Then by Lemma \ref{proof-lem-1}, $$k_{D'}(z_1',z_2')\leq M k_D(z_1,z_2)\leq 2M \frac{|z_1-z_2|}{d_D(z_1)}\leq 1,$$ which shows that $$\frac{1}{2}\frac{|z_1'-z_2'|}{d_{D'}(z_1')}\leq k_D(z_1',z_2')\leq M k_D(z_1,z_2)\leq 2M\frac{|z_1-z_2|}{d_D(z_1)}.$$ Hence Lemma \ref{lem-1} shows that
\begin{equation}\label{thm-1-2-proof1}|z_1'-z_2'|\leq 4M\frac{|z_1-z_2|}{d_D(z_1)}d_{D'}(z_1')\leq 4MM_1|z_1-z_2|. \end{equation}

Next we consider the case $|z_1-z_2|> \frac{1}{2M}d_D(z_1)$. We let $z\in \partial D$ be such that $|z_1-z|\leq 2 d_D(z_1)$. If $|z_1-z|\leq \frac{1}{2}|z_2-z|$, then $$|z_1-z_2|\geq |z_2-z|-|z_1-z|\geq \frac{1}{2}|z_2-z|,$$ and so Lemma \ref{lem-1}  yields
\begin{equation}\label{thm-1-2-proof2}|z_1'-z_2'|\leq |z_1'-z'|+|z_2'-z'|\leq M_1 d_D(z_1)+ 2(2L+M_1)|z_2-z|\end{equation}\begin{eqnarray*}\leq 2(MM_1+4L+2M_1)|z_1-z_2|.\end{eqnarray*}
On the other hand, if  $|z_1-z|\geq \frac{1}{2}|z_2-z|$, then by Lemma \ref{lem-1} we have  \begin{equation}\label{thm-1-2-proof3}|z_1'-z_2'|\leq |z_1'-z'|+|z_2'-z'|\leq M_1 d_D(z_1)+ 2(2L+M_1)|z_2-z|\end{equation}\begin{eqnarray*}\leq M_1 d_D(z_1)+ 4(2L+M_1)|z_1-z|\leq 2M(9M_1+16L)|z_1-z_2|.\end{eqnarray*} By taking $M'= 2M(9M_1+16L)$ we see from \eqref{thm-1-2-proof1}, \eqref{thm-1-2-proof2} and \eqref{thm-1-2-proof3} that \eqref{thm-1-2} holds. Hence the proof of Theorem \ref{thm1.4} is complete.\qed

\medskip

\begin{remark}\label{remark}\begin{enumerate}
\item In Theorem \ref{thm1.4}, the hypothesis ``$f$ is FQC" alone does not imply the conclusion ``$f$ is bilipschitz". As an example, we consider the radial power map $f_{\alpha}: \mathbb{B}\to \mathbb{B}$ with $f_{\alpha}(x)=|x|^{\alpha-1}x$ and $\alpha\geq 1$. By \cite[6.5]{Vai6-0} we see that $f_{\alpha}$ is a FQC map  and $f_{\alpha}|_{\partial \mathbb{B}}$ is the identity on the boundary,  but $f_{\alpha}$ is not bilipschitz (see \cite[6.8]{Vai5}).
    \medskip
\item If the boundary of $D$ is not HD, then ``$f$ being QH" does not always imply that ``$f$ is bilipschitz". We still consider the radial power map $f_{\alpha}: E\setminus\{0\}\to E\setminus\{0\}$ with $f_{\alpha}(x)=|x|^{\alpha-1}x$ and $\alpha\geq 1$. On one hand, the domain $E\setminus\{0\}$ has only two boundary components: $\{0\}$ and $\{\infty\}$, and so the boundary is not HD. On the other hand, $f$ is $\alpha$-QH (see \cite[5.21]{Vai5}) and it is the identity on the boundary. But it is not bilipschitz.
    \medskip
%
%

\end{enumerate}
\end{remark}
  \medskip
%
%

\medskip
\mysubsection{\bf The proof of Theorem \ref{thm1.1}}

Given $x\in D=D'$, let $z'\in \partial D'$ be such that $d_{D'}(x')\geq \frac{1}{2}|x'-z'|$. Then Lemma \ref{lem-1} yields
$$d_{D'}(x')\geq \frac{1}{4(2L+M_1)}|x-z|\geq \frac{1}{4(2L+M_1)}d_D(x).$$

\noindent Let $z_1\in \partial D$ be such that $|x-z_1|\leq 2 d_D(x)$. Then it follows from Lemma \ref{lem-1} that $$|x-x'|\leq |x-z_1|+|x'-z_1|\leq (2+M_1)d_D(x).$$

\noindent Hence by Lemma \ref{pre-lem-1} we see that $$k_D(x,x')\leq c'\log\Big(1+\frac{|x-x'|}{\min\{d_D(x),d_D(x')\}}\Big)\leq c'\log\big(1+4(2+M_1)(2L+M_1)\big).$$\qed

\medskip

{\bf Acknowledgement.} This research was finished when the first author was an academic visitor
in Turku University and the first author was supported by the Academy of Finland grant of Matti Vuorinen with the
Project number 2600066611. She thanks
Department of Mathematics in Turku University for hospitality.

\end{document}